\documentclass{amsart}

\usepackage{amsmath}
\usepackage{verbatim}
\usepackage{amssymb}

\numberwithin{equation}{section}

\newtheorem{theorem}{Theorem}
\newtheorem{lemma}{Lemma}
\newtheorem{corollary}{Corollary}

\theoremstyle{definition}
\newtheorem{rem}{Remark}

\newcommand{\dd}{\mathrm d}
\newcommand{\DD}{\mathrm D}
\newcommand{\cD}{\mathcal D}
\newcommand{\cC}{\mathcal C}
\newcommand{\cL}{\mathcal L}

\newcommand{\cK}{\mathcal K}
\newcommand{\cF}{\mathcal F}
\newcommand{\IR}{\mathbb R}
\newcommand{\IN}{\mathbb N}
\newcommand{\tD}{\text D}
\newcommand{\tN}{\text N}
\newcommand{\vp}{\varphi}
\newcommand{\uu}{\underline u}
\newcommand{\uf}{\underline f}

\DeclareMathOperator{\Span}{span}
\DeclareMathOperator{\meas}{meas}
\DeclareMathOperator{\diag}{diag}

\title[Well-Posedness of a hyperbolic type integro-differential equation]
{Well-posedness of an integro-differential equation with positive type kernels 
modeling fractional order viscoelasticity}

\author[F.~Saedpanah]{Fardin Saedpanah}

\address{Department of Mathematics, 
University of Kurdistan, P. O. Box 416, 
Sanandaj, Iran}

\email{f.saedpanah@uok.ac.ir\\
           f\_saedpanah@yahoo.com}

\thanks{Research supported by University of Kurdistan under grant (4/1238).}



\begin{document}

\begin{abstract}
A hyperbolic type integro-differential equation with two weakly 
singular kernels is considered together with 
mixed homogeneous Dirichlet and non-homogeneous 
Neumann boundary conditions. 
Existence and uniqueness of the solution is proved by means of 
Galerkin's method. Regularity estimates are 
proved and the limitations of the regularity are discussed. 
The approach presented here is also used to prove 
regularity of any order for models with smooth kernels, 
that arise in the theory of linear viscoelasticity, 
under the appropriate assumptions on data. 
\end{abstract}

\date{October 24, 2013}

\maketitle
 
%---------------------------------------------------------------------%
\section{Introduction}
%---------------------------------------------------------------------%
We study the model problem \eqref{strongform}, which is a 
hyperbolic type integro-differential equation with two weakly 
singular kernels of Mittag-Leffler type. 
This problem arises as a model for fractional 
order viscoelasticity. The fractional order viscoelastic 
model, that is, the linear viscoelastic model with 
fractional order operators in the constitutive equations, 
is capable of describing the behavior of many viscoelastic materials 
by using only a few parameters. 

A perfectly elastic material does not exist since in reality: 
inelasticity is always present. 
This inelasticity leads to energy dissipation or damping. 
Therefore, for a wide class of materials it is not sufficient 
to use an elastic constitutive model to capture the 
mechanical behaviour.
In order to replace extensive experimental tests by numerical
simulations there is a need for an accurate material model.
Therefore viscoelastic constitutive models have frequently 
been used to simulate the
 time dependent behaviour of polymeric materials. 
The classical linear viscoelastic models that use integer order 
time derivatives in the constitutive laws, 
require an excessive number of parameters to accurately 
predict observed material behaviour, 
see e.g., \cite{AdolfssonEnelundOlsson} and 
\cite{RiviereShawWhiteman} for examples and more references. 
These models describe  e.g., polymeric materials such as 
natural and synthetic rubber, and require a large number of exponential 
(smooth) kernels to describe the behavior of the materials. 

Bagley and Torvik \cite{BagleyTorvik} used fractional derivatives 
to construct stress-strain relationships for viscoelastic materials. 
The advantage of this approach is that very few empirical 
parameters are required. 
When this fractional derivative model of viscoelasticity is incorporated directly into
 the structural equations a time differential equation of non-integer order higher
 than two is obtained. One consequence of this is that initial conditions of
fractional order higher than one are required. The problems with initial conditions
 of fractional order have been discussed by Enelund and Olsson 
 \cite{EnelundOlsson}, see also references therein. 
To avoid the difficulties with fractional order initial conditions some
alternative formulations of the fractional derivative viscoelastic model are used in
structural modeling. The formulation that we use, is based on a convolution integral formulation with weakly singular fractional order kernels 
of Mittag-Leffler type, see \cite{Stig4}, 
\cite{EnelundJosefson}, and \cite{EnelundOlsson}. 
For other formulations, that involves fractional integral 
operators rather than fractional derivative operators, 
or uses internal variables, see 
\cite{EnelundFenanderOlsson}, \cite{EnelundLesieutre} and 
\cite{EnelundMahlerRunessonJosefson}. 

Another formulations of fractional derivative viscoelastic model 
can be in terms of so-called diffusive representation. 
It is a different approach, from the convolution integral formulation 
that is presented here, and it was introduced for numerical 
simulation of complex dynamics in \cite{MontsenyAudounetMbodge}. 
Based on diffusive representation 
of the fractional integral/differential operators, for links between 
these concepts see \cite{Matignon1998}, the output solution 
is represented in terms of a so-called diffusive symbol  
and a state function, that is a solution of an ordinary 
differential equation in time. The state function is called the 
diffusive representation of the input. 
Diffusive realizations of the fractional integral/differential operators, 
using the Laplace transform of their kernels, avoids the hereditary 
behavior of such operators. 
This means that, for time domain discretization methods, we need 
only the previous time step to update the integral at each time step. 
For more references and applications of this method see 
\cite{HaddarMatignon}, \cite{DeuMatignon} and references therein. 
%\cite{LaudebatBidanMontseny}. 

The fractional order kernels are the only mean to get a correct 
representation for the storage and loss moduli, and to have well-posed 
identification problem for many viscoelastic materials. 
Important properties of such kernels are integrability and 
completely monotonicity, that (as a consequence of dissipation) 
implies to be positive type. In fact, these kernels interpolate between smooth (exponential) kernels and weakly singular kernels, that are 
singular at origin but integrable on finite time intervals, i.e., belong to 
$L_{1,loc}(\IR^+)$.  
A chief example is 
$\beta(t)=\frac{1}{\Gamma(\alpha)}\frac{1}{t^{1-\alpha}},\ \alpha \in (0,1)$. 
For more details and examples see 
\cite{CiambellaPaoloneVidoli} and references therein.  
This is the reason for introducing kernels
of Mittag-Leffler type or fractional operators.
In \cite{AdolfssonEnelundOlsson} and \cite{EnelundOlsson} 
it is shown that the classical
viscoelastic model based on exponential kernels can describe
the same viscoelastic behaviour as the fractional model if the
number of kernels tend to infinity. 

In fractional order viscoelastic models 
the whole strain history must be saved and included in each time 
step that is due to the non-locality of the fractional order 
integral/differential operators. The most commonly used algorithms 
for this integration are based on Lubich convolution quadrature 
\cite{Lubich} for fractional order operators, 
see also \cite{LopezLubichSchadle} for an improved version. 
For examples of the application of this approch to overcome 
the problem with the growing amount of data, that has to be stored 
and used in time stepping methods, see  
\cite{Stig4}, \cite{LubichSloanThomee} and \cite{McLeanThomeeWahlbin}. 
For analysis and numerical solution 
of integro-differential equations and related problems, 
from the extensive literature, see e.g.,   
\cite{StigFardin}, \cite{McLeanThomee2010}, \cite{RiviereShawWhiteman}, 
and their references. 

Existence, uniqueness, and regularity of the solution of 
models with exponential kernels can be adapted from, 
e.g., \cite{Dafermos}, where an abstract Volterra equation, 
as an abstract model for equations of linear viscoelasticity, 
has been considered. 
See also \cite{Fichera1979} for another paper dealing with 
well-posedness of problems in linear viscoelasticity 
with smooth kernels. 
Existence, uniqueness and regularity of a parabolic type integro-differential 
equation has been studied in \cite{McLeanThomee} by means of 
Fourier series. 
One may also see \cite{DeschFasanga}, where the 
theory of analytic semigroups is used in terms of 
interpolation spaces to solve a boundary value problem in 
linear viscoelasticity. 
Well-posedness of an integro-differential equation, 
a model from dynamic linear viscoelasticity with exponential kernels 
and first order spatial operator in the convolution integral, 
has been studied in \cite{McLaughlinThomasYoon}, 
by means of Galerkin approximation method. 
However, the mixed homogeneous Dirichlet and non-homogeneous 
Neumann boundary condition, that is important for practitioners, 
has not been considered. 

In a previous work \cite{StigFardin}, well-posedness and 
regularity of the problem \eqref{strongform2}, 
which is a simplified form (synchronous viscoelasticity) 
of the model problem \eqref{strongform}, 
was studied in the framework of the semigroup of linear operators. 
The drawback of the framework is that this does not 
admit non-homogeneous Neumann boundary condition, while in 
practice mixed homogeneous Dirichlet and non-homogeneous 
Neumann boundary conditions are of special interest. 

Here we consider the model problem \eqref{strongform}, 
which is a hyperbolic type integro-differential equation with two 
weakly singular kernels of Mittag-Leffler type, 
and it is the convolution integral formulation of the 
constitutive equation system  \eqref{constitutiveequation}.  
The mixed homogeneous Dirichlet and non-homogeneous 
Neumann boundary condition has been considered, 
and we investigate existence, uniqueness and regularity 
of the solution of the model problem by 
means of the Galerkin approximation method. 
We also extend the presented approach so that regularity 
of any order of the solution for the models with smooth kernels 
can be proved.

In the sequel, in $\S2$ we describe the construction of the 
model problem \eqref{strongform}.   
In $\S3$ we define a weak (generalized) solution and, 
using Galerkin's method, we prove existence and 
uniqueness of the weak solution of the problem. 
Finally, in $\S4$ we study regularity of the solution and 
limitations for higher regularity. 
We also show that higher regularity of any order of the 
solution of models with smooth kernels can be achieved. 

%------------------------------------------------------------------------------%
\section{Fractional order viscoelasticity}
%------------------------------------------------------------------------------%
Let $\sigma_{ij}$, $\epsilon_{ij}$ and $u_i$ denote, respectively, 
the usual stress tensor, strain tensor and displacement vector. 
We recall that the linear strain tensor is defined by, 
\begin{align*}
  \epsilon_{ij}
  = \frac12 \Big( 
  \frac{\partial u_i}{\partial x_j} 
  + \frac{\partial u_j}{\partial x_i} \Big) .
\end{align*}
We recall that the simplest fractional derivative model of viscoelasticity, 
so-called fractional Zener model, is 
\begin{equation}   \label{Zenermodel}
    \sigma(t) + \tau^{\alpha} D_t^{\alpha}\sigma(t)
      = E_\infty \epsilon(t) 
       + E\tau^{\alpha} D_t^{\alpha}\epsilon(t),
\end{equation}
where $\tau$ is the relaxation time, $\alpha$ is the fractional order 
of differentiation, and $E$, $E_\infty$ are the instantaneous (unrelaxed) 
and long-time (relaxed) modulus, respectively. 
This model has been shown to describe the actual weak frequency 
dependence, of the complex modulus, for rather a wide class of 
engineering materials, 
see \cite{BagleyTorvik},  \cite{EnelundOlsson} and \cite{Pritz1996} 
for more details. 

%For convenience when extending the fractional Zener model to 3D,  
%the following decomposition of the stress and strain tensors is introduced
%Since for many materials, e.g., viscoelastic materials, the stress 
%depends on the entire strain history, it is convenient to introduce 
%the decompostion
With the decompositions
\begin{align*}
  s_{ij}=\sigma_{ij}-\tfrac13 \sigma_{kk}\delta_{ij}, \quad
  e_{ij}=\epsilon_{ij}-\tfrac13 \epsilon_{kk}\delta_{ij},
\end{align*}
the constitutive equations, fractional Zener models, 
are formulated as, see \cite{EnelundJosefson},
\begin{equation}   \label{constitutiveequation}
  \begin{split}
    s_{ij}(t) &+ \tau_1^{\alpha_1} D_t^{\alpha_1}s_{ij}(t)
      = 2G_\infty e_{ij}(t) + 2G\tau_1^{\alpha_1} D_t^{\alpha_1}e_{ij}(t),\\
    \sigma_{kk}(t) &+ \tau_2^{\alpha_2} D_t^{\alpha_2}\sigma_{kk}(t)
      = 3K_\infty \epsilon_{kk}(t) 
       + 3K\tau_2^{\alpha_2} D_t^{\alpha_2}\epsilon_{kk}(t),
  \end{split}
\end{equation}
with initial conditions
\begin{align*}
  s_{ij}(0+)=2Ge_{ij}(0+), \quad
  \sigma_{kk}(0+)=3K\epsilon_{kk}(0+),
\end{align*}
meaning that the initial response follows Hooke's elastic law. 
Here $G,\,K$ are the instantaneous (unrelaxed) shear and bulk modulus, 
and $G_\infty,\,K_\infty$ are the long-time (relaxed) shear and bulk modulus, 
respectively. 
Note that we have two relaxation times, $\tau_1,\tau_2>0$, 
and fractional orders of differentiation, $\alpha_1,\,\alpha_2\in(0,1)$,
 where the fractional order derivative is defined by, \cite{Podlubny}, 
\begin{align*}
  D_t^\alpha f(t)
  =  D_t   D_t^{-(1-\alpha)} f(t)
  = D_t \frac{1}{\Gamma(1-\alpha)}  \int_0^t (t-s)^{-\alpha}f(s)\,ds .
\end{align*}
The constitutive equations \eqref{constitutiveequation} 
can be solved for $\sigma$ by means of Laplace transformation, 
%\cite{EnelundJosefson} and 
\cite{EnelundOlsson}:
\begin{align*}
  \begin{aligned}
    s_{ij}(t)&=2G\Big(e_{ij}(t) 
       - \frac {G-G_\infty}G \int_0^t \theta_1(t-s)e_{ij}(s) \,ds \Big), \\
    \sigma_{kk}(t) &=3K\Big( \epsilon_{kk}(t)  
       - \frac {K-K_\infty}K \int_0^t \theta_2(t-s)\epsilon_{kk}(s)
         \,ds\Big), 
  \end{aligned}
\end{align*}
where, for $i=1,2$,
\begin{align*}
  \theta_i(t) = - \frac{d}{dt} E_{\alpha_i}
  \Big(-\Big(\frac{t}{\tau_i}\Big)^{\alpha_i}\Big),\quad
  E_{\alpha_i}(z)=\sum_{n=0}^\infty \frac{z^n}{\Gamma(1+n\alpha_i)},
\end{align*}
and $E_{\alpha_i}$ is the Mittag-Leffler function of order 
$\alpha_i$, \cite{Bateman}. 
Then we define parameters 
$\gamma_i$, and the Lam\'e constants $\mu$, $\lambda$, 
\begin{align*}
   \gamma_1=\frac {G-G_\infty}G, 
    \quad \gamma_2 =  \frac {K-K_\infty}K, 
    \quad \mu = G ,  \quad  \lambda = K-\tfrac23G.
\end{align*}
We recall that, due to dissipation, we need the assumtions, 
for $\alpha_i \in (0,1),\  i=1,2$, 
\begin{equation*}
  K>K_\infty>0, \quad G>G_\infty>0, \quad \tau_i>0,
\end{equation*}  
and therefore we have $0<\gamma_i<1,\ i=1,2$, 
see \cite{BagleyTorvik1986}, 
and also \cite{AdolfssonEnelundOlsson}, \cite{EnelundJosefson} 
for examples. 

We also define $\beta_i=\gamma_i\theta_i$,
and the constitutive equations become
\begin{align*} 
  \begin{split}
    \sigma_{ij}(t)
      &=\Big(2\mu\epsilon_{ij}(t)+\lambda\epsilon_{kk}(t)\delta_{ij} \Big)
        -2\mu\int_0^t\!\beta_1(t-s) 
	\Big(\epsilon_{ij}(s)-\tfrac13 \epsilon_{kk}(s)\delta_{ij}
	 \Big)\,ds\\
      &\quad -\frac{3\lambda+2\mu}{3}
         \int_0^t\!\beta_2(t-s)\epsilon_{kk}(s)\delta_{ij}\,ds.
  \end{split}     
\end{align*}
The kernels are weakly singular, i.e., singular at the origin but integrable, 
for $i=1,\,2$:
\begin{align}  \label{beta}
  \begin{split}
    \beta_i(t)&=-\gamma_i  \frac{d}{dt}E_{\alpha_i}
      \Big(-\Big(\frac{t}{\tau_i}\Big)^{\alpha_i}\Big)
      =\gamma_i\frac{\alpha_i}{\tau_i} \Big(\frac{t}{\tau_i}\Big)^{-1+\alpha_i}
      E_{\alpha_i}'  \Big(-\Big(\frac{t}{\tau_i}\Big)^{\alpha_i}\Big)\\
    &\approx C t^{-1+\alpha_i}, \ \text{ $t\to 0$} ,
  \end{split}
\end{align}
and we note the properties
\begin{align}  \label{betaproperties}
  \begin{aligned}
    \beta_i(t)&\ge 0, \\ 
      \|\beta_i\|_{L_1(\IR^+)}&=\int_0^\infty\!
        \beta_i(t)\,dt
   =\gamma_i \Big( E_{\alpha_i} (0)- E_{\alpha_i}(\infty)\Big)=\gamma_i <1.
  \end{aligned}
\end{align}

The equations of motion are
\begin{align}  \label{motion}
  \begin{aligned}
  &\rho u_{i,tt}-\sigma_{ij,j} =f_i ,  \quad &&\text{in } \Omega, \\
   &u_i=0 , \quad &&\text{on } \Gamma_\text{D}, \\
   &\sigma_{ij}n_j=g_i , \quad &&\text{on } \Gamma_\text{N},
  \end{aligned}
\end{align}
where $u$ is the displacement vector, $\rho$ is the (constant) mass density,
$f$ and $g$ represent, respectively, the volume and surface loads. 
We let $\Omega \subset\mathbb{R}^d, \, d=2,3$, be a bounded 
polygonal domain with boundary 
$\Gamma_\text{D}\cup\Gamma_\text{N}=\partial\Omega,\,
\Gamma_\text{D}\cap\Gamma_\text{N}=\varnothing$ 
and $\meas(\Gamma_\text{D})\neq 0$. 
We set
\begin{align} \label{Ai}
  \begin{split}
   (Au)_i&=-\big(2\mu\epsilon_{ij}(u)
       +\lambda\epsilon_{kk}(u)\delta_{ij}\big)_{\!,j}\,,\\
   (A_1u)_i&=-2\mu\big(\epsilon_{ij}(u)
       -\tfrac13\epsilon_{kk}(u)\delta_{ij}\big)_{\!,j}\,,\\
   (A_2u)_i&=-\frac{3\lambda+2\mu}{3}
       \big(\epsilon_{kk}(u)\delta_{ij}\big)_{\!,j}\,,
  \end{split}
\end{align}
and clearly we have $A=A_1+A_2$. 
Now, we write the equations of motion \eqref{motion} in the strong form, (we denote time derivatives with '$\cdot$'),
\begin{align}   \label{strongform}
  \begin{aligned}
    &\rho \ddot u(x,t)+Au(x,t)\\
    &\qquad-\sum_{i=1}^2\int_0^t\!\beta_i(t-s)A_iu(x,s)\,ds=f(x,t)
        \quad&& \textrm{in} \;\,\Omega\times (0,T),\\
    &u(x,t)=0 \quad&& \textrm{on}\;\Gamma_\text{D}\times (0,T),\\
    &\sigma(u;x,t)\cdot n=g(x,t)
        \quad&&\textrm{on}\;\Gamma_\text{N}\times (0,T),\\
    &u(x,0)=u^0(x)\quad&&\textrm{in}\;\,\Omega,\\
    &\dot u(x,0)=v^0(x)\quad&&\textrm{in}\;\,\Omega. 
  \end{aligned}
\end{align}

%------------------    Remark    -----------------%
\begin{rem}
If we make the simplifying
assumption (synchronous viscoelasticity), 
that is, when all elastic modulus at each 
material point have the same relaxation behavior:
\begin{align*}
 \quad \alpha=\alpha_1=\alpha_2, 
 \quad  \tau=\tau_1=\tau_2,\quad \theta=\theta_1=\theta_2,
\end{align*}
we may define $\gamma=\gamma_1=\gamma_2$, so that 
$\beta=\beta_1=\beta_2$. 
Then the strong form of the equations of motion is
\begin{align}   \label{strongform2}
  \begin{aligned}
    \rho \ddot u(x,t)
    &+Au(x,t)
    -\int_0^t\!\beta(t-s)Au(x,s)\,ds=f(x,t)
      \qquad \textrm{in} \;\,\Omega\times (0,T),
  \end{aligned}
\end{align}
together with the boundary and initial conditions in 
\eqref{strongform}. We note that \eqref{strongform2} 
is the strong form of the equation of motion of the simplest 
fractional model of viscoelasticity \eqref{Zenermodel}, 
in the convolution integral formulation. 
And we note that \eqref{strongform} is also valid for small 
deformations, that is true deviatoric and bulk parts. 

Well-posedness and regularity of the simplified problem 
\eqref{strongform2} was studied in \cite{StigFardin}, 
in the framework of the semigroup of linear operators.  
The drawback of the framework is that this does not 
admit non-homogeneous Neumann boundary condition. 
More details and examples of the simplified problem  
\eqref{strongform2} can be found in  
\cite{AdolfssonEnelund2003}, \cite{Stig3}, \cite{Stig4}, 
\cite{EnelundJosefson}, 
\cite{EnelundMahlerRunessonJosefson}, 
\cite{EnelundOlsson}, and \cite{SchmidtGaul}, e.g., 
bar, beam and plain starin in 2D can be found in 
\cite{EnelundJosefson}, 
\cite{EnelundMahlerRunessonJosefson}, and 
\cite{EnelundOlsson}. 
\end{rem}

%-------------------------------------------------------------------------%
\section{Existence and uniqueness}
%-------------------------------------------------------------------------%
In this section we prove existence and uniqueness of a weak 
solution of \eqref{strongform} using Galerkin's method, 
in a similar way for hyperbolic PDE's in 
\cite{DautrayLions}, \cite{Evans}. 
To this end, we first formulate the weak form of the model 
problem \eqref{strongform}. 
Then we introduce the Galerkin approximation of a 
weak solution of \eqref{strongform} in a classical way, and 
we obtain a priori estimates for approximate solutions. 
These will be used to construct a weak solution, and then uniqueness will be verified.

%-------------------------------------------------------------------------%
\subsection{Weak formulation}
%-------------------------------------------------------------------------%
We define the bilinear form (with the usual summation convention)
\begin{equation*}
  a(u,v)=\int_{\Omega}\!\big(
  2\mu\epsilon_{ij}(u)\epsilon_{ij}(v)
   +\lambda\epsilon_{ii}(u)\epsilon_{jj}(v)\big)\,dx,
   \quad \forall u,v\in V,
\end{equation*}
which is well-known to be coercive. In a similar way, 
corresponding to $A_i,\ i=1,\ 2$, the bilinear forms $a_i(u,v)$ are defined. 
We introduce the function spaces
$H=L_2(\Omega)^d,\,H_{\Gamma_\tN}=L_2(\Gamma_\tN)^d,\,$ 
and 
$V=\{v\in H^1(\Omega)^d:v\!\!\mid_{\Gamma_\tD}=0\}$. 
We denote the norms in $H$ and $H_{\Gamma_\tN}$ by
$\|\cdot\|$ and $\|\cdot\|_{\Gamma_\tN}$,
 respectively, and we equip $V$ with the inner product 
$a(\cdot,\cdot)$ and norm 
$\|v\|_V^2=a(v,v)$. 
%It can be shown that $A,\,A_1$ and $A_2$ 
%are self-adjoint, positive definite linear operators (not bounded). 
We note that, for $v\in V$,
\begin{equation}   \label{ai}
  a_i(v,v)\le \|v\|_V^2.
\end{equation}

Now we define  a weak solution to be a function $u=u(x,t)$ 
that satisfies
\begin{align}   \label{weaksolution1}
  &u\in L_2((0,T);V),\quad \dot u\in L_2((0,T);H),
    \quad \ddot u\in L_2((0,T);V^*),\\ \label{weaksolution2}
  &\rho \langle \ddot u(t),v\rangle + a(u(t),v)
     -\sum_{i=1}^2 \int_0^t\! \beta_i(t-s) a_i(u(s), v) \,ds \\
  &\nonumber\qquad\qquad\qquad\qquad= (f(t),v) + (g(t),v)_{\Gamma_\text{N}},
    \quad \forall v \in V ,\,\,\text{a.e.}\,\,t\in (0,T),\\
      \label{weaksolution3}
  &u(0)=u^0,\quad \dot u(0)=v^0.
\end{align}
Here $(g(t),v)_{\Gamma_\text{N}}=\int_{\Gamma_\text{N}}\!g(t)\cdot v\,dS$,
 and $\langle\cdot,\cdot\rangle$ denotes the pairing of $V^*$ and $V$.
We note that \eqref{weaksolution1} implies, 
by a classical result for Sobolev spaces, that 
$u\in \cC([0,T];H),\,\dot u\in\cC([0,T];V^*)$ 
so that the initial conditions \eqref{weaksolution3} make sense for 
$u^0\in H,\,v^0\in V^*$.

%-----------------------------------------------------------------------%
\subsection{Galerkin approximations}
%-----------------------------------------------------------------------%
Let $\{(\lambda_j,\vp _j)\}_{j=1}^\infty$ be the eigenpairs 
of the weak eigenvalue problem
\begin{equation}   \label{weakeigenvalue}
  a(\vp,v)=\lambda(\vp,v),\quad \forall v\in V.
\end{equation}
It is known that $\{\vp _j\}_{j=1}^\infty$ can be chosen to be an ON-basis in 
$H$ and an orthogonal basis for $V$.

Now, for a fixed positive integer $m\in\IN$, we seek a function of the form
\begin{equation}   \label{u_m}
  u_m(t)=\sum_{j=1}^m d_j(t)\vp _j
\end{equation}
to satisfy
\begin{equation}   \label{weakGalerkin}
  \begin{split}
    \rho (\ddot u_m(t),\vp_k&) + a(u_m(t),\vp_k)
    -\sum_{i=1}^2 \int_0^t\! \beta_i(t-s) a_i(u_m(s), \vp_k)\,ds\\
    &= (f(t),\vp_k) + (g(t),\vp_k)_{\Gamma_\text{N}},
    \quad k=1,\,\dots,\,m ,\,t\in (0,T),
  \end{split}
\end{equation}
with initial conditions
\begin{equation}   \label{weakGalerkininitial}
  u_m(0)=\sum_{j=1}^m(u^0,\vp _j)\vp _j,
    \quad \dot u_m(0)=\sum_{j=1}^m(v^0,\vp _j)\vp _j.
\end{equation}

\begin{lemma}
For each $m\in \IN$, there exists a unique function $u_m$ of the form
\eqref{u_m} satisfying \eqref{weakGalerkin}-\eqref{weakGalerkininitial}. 
\end{lemma}
%-----------------------------     PROOF       -----------------------------%
\begin{proof}
Using \eqref{u_m} and the fact that $\{\vp _j\}_{j=1}^\infty$ is an 
ON-basis for $H$ and a solution of the eigenvalue problem 
\eqref{weakeigenvalue}, we obtain from
\eqref{weakGalerkin} that,
\begin{equation}   \label{SystemweakGalerkin}
  \begin{split}
    \rho \ddot d_k(t)+\lambda_k d_k(t) &- \sum_{j=1}^m \sum_{i=1}^2 
      a_i(\vp _j,\vp _k) (\beta_i*d_j)(t)\\
    &= f_k(t) + g_k(t),
      \quad k=1,\,\dots,\,m ,\,t\in (0,T),
  \end{split}
\end{equation}
where $*$ denotes the convolution, and 
$f_k(t)=(f(t),\vp _k),\,g_k(t)=(g(t),\vp _k)_{\Gamma_\text{N}}$. 
This is a linear system of second order ODE's
 with initial conditions
\begin{equation}   \label{SystemweakGalerkininitial}
  d_k(0) = (u^0,\vp_k),\quad \dot d_k(0) = (v^0,\vp_k),
    \quad k=1,\dots,m.
\end{equation}
The Laplace transform can be used, for example, to find the unique solution
of the system. 

We note that, the Laplace transform of the Mittag-Leffler function is, 
see e.g., \cite{Podlubny}, 
\begin{equation*}
  \cL(E_\alpha(a t^\alpha))=\frac{s^{\alpha-1} }{s^\alpha - a}\,,
  \quad {\rm Re}(s) > |a|^{1/\alpha},
\end{equation*}
and therefore, for the kernels $\beta_i,\ i=1,2$, 
defined in \eqref{beta}, we have
\begin{equation*}
  \begin{split}
    \hat{\beta}_i(s)
    &=\cL(\beta_i(t))
     =-\gamma_i s \cL(E_{\alpha_i}(-\tau_i^{-\alpha_i}t^{\alpha_i}))
     +\gamma_i E_{\alpha_i}(0)\\
    &=-\gamma_i s \frac{s^{\alpha_i-1} }
     {s^{\alpha_i}+\tau_i^{-\alpha_i}}+\gamma_i
     =\gamma_i - \gamma_i \frac{s^{\alpha_i} }
     {s^{\alpha_i}+\tau_i^{-\alpha_i}}\\
    &=\frac{\gamma_i}{(\tau_i s)^{\alpha_i} +1}<1\ ,
     \quad {\rm Re}(s) > \tau_i^{-1} .
  \end{split}
\end{equation*}

Now taking the Laplace transform of \eqref{SystemweakGalerkin}  
we get (we use an over-hat for the Laplace transform),
\begin{equation}   \label{SystemLaplaceTransform}
  \begin{split}
    (s^2\rho+\lambda_k)\hat d_k(s)
      &-\sum_{j=1}^m\sum_{i=1}^2 a_i(\vp _j,\vp _k)
        \frac{\gamma_i}{(\tau_i s)^{\alpha_i} +1}\hat d_j(s)\\
    &=\hat f_k(s)+\hat g_k(s)+s\rho d_k(0)+\rho \dot d_k(0),\\
    & \qquad k=1,\dots,m,\  {\rm Re}(s) > \frac{1}{\min \{\tau_1,\tau_2\}},
  \end{split}
\end{equation}
that can be written in the matrix form
\begin{equation*}
  Q \hat D =\hat F + P.
\end{equation*}
Here
\begin{equation*}
  \begin{split}
    &Q(s)=(Q_{j,k}(s))_{j,k=1}^m
    =\left\{\begin{array}{ll}
     \rho s^2 +\lambda_k 
     -\sum_{i=1}^2a_i(\vp_k,\vp_j)\frac{\gamma_i}{(\tau_i s)^{\alpha_i} +1}, 
     & j=k,\\
     -\sum_{i=1}^2a_i(\vp_k,\vp_j)\frac{\gamma_i}{(\tau_i s)^{\alpha_i} +1},
     &j\neq k,
    \end{array}\right.\\
    &\hat D(s)=(\hat d_k(s))_{k=1}^m, \quad 
    \hat F(s)=(\hat f_k(s)+\hat g_k(s))_{k=1}^m, \quad  
    P(s) =(s\rho d_k(0)+\rho\dot d_k(0))_{k=1}^m,
  \end{split}
\end{equation*}
and we note that the entries of matrix $A$ are analytic. 

For the extreme cases $\alpha_i=0,1,\ i=1,2$, components of 
$\hat D=Q^{-1}(\hat F+P)$ are proper rational functions, and it is well-known 
that the inverse Laplace transform is uniquely computable, using 
partial fractions expansion.  
Therefore $\alpha_1,\alpha_2\in(0,1)$ interpolates between these two cases, 
for which the inverse Laplace transform is uniquely computable. 
Indeed, having $\{s_j\}_{ j=1}^J$ finite poles for $\hat D=Q^{-1}(\hat F+P)$, 
we denote 
\begin{equation*}
  \overline a=\max \Big\{\frac{1}{\min \{\tau_1,\tau_2\}}, 
      {\rm Re}(s_j), \dots,{\rm Re}(s_J) \Big\}, \quad
  \underline a=\min \big\{{\rm Re}(s_j), \dots,{\rm Re}(s_J) \big\}.
\end{equation*}
Therefore $Q^{-1}(\hat F+P)$ is analytic for ${\rm Re}(s)>\overline a$,  
and the inverse Laplace transform is uniquely computable, 
 \cite[Theorem 8.5]{Folland}, and for the contour 
 $\underline a \leq {\rm Re}(s) \leq \overline a$ one can use the residue theorem. 

Hence, there is a unique solution for the linear system 
\eqref{SystemweakGalerkin}, 
and this completes the proof.
\end{proof}

For our analysis below to manipulate the non-homogeneous 
Neumann boundary condition, 
recalling \eqref{beta} and \eqref{betaproperties}, 
we define the functions
\begin{align*}   %\label{xi}
  \xi_i(t)=\gamma_i-\int_0^t\!\beta_i(s)\,ds
    =\int_t^\infty\!\beta_i(s)\,ds
    =\gamma_i E_{\alpha_i}(t),\quad i=1,\,2,
\end{align*}
and it is easy to see that
\begin{align}   \label{xiproperties}
  \begin{split}
    D_t\xi_i(t)=-\beta_i(t)<0,\quad
    \xi_i(0)=\gamma_i,\quad \lim_{t\to\infty}\xi_i(t)=0,\quad
    0<\xi_i(t)\le\gamma_i.
  \end{split}
\end{align}
Besides, $\xi_i$ are completely monotonic functions, that is,
\begin{align*}
  (-1)^jD_t^j\xi_i(t)\ge 0,\quad t\in (0,\infty),\,j\in\IN,
\end{align*}
since the Mittag-Leffler functions $E_{\alpha_i},\,\alpha_i\in[0,1]$ are
 completely monotonic.
Consequently, an important property of $\xi_i,i=1,\,2$, is that, they are positive type kernels, 
that is, they are continuous and, 
for any $T\ge 0$, satisfy
\begin{align}   \label{positivetype}
  \int_0^T\!\int_0^t\!\xi_i(t-s)\phi(t)\phi(s)\,ds\,dt\ge 0,
    \quad \forall\phi\in \cC([0,T]).
\end{align}
For more details on these concepts and their properties see, e.g., 
\cite{Bateman} 
%, \cite{RenardyHrusaNohel} 
and \cite{Widder}.

Our plan is to send $m\to \infty$ and prove existence of a 
weak solution of \eqref{weaksolution1}-\eqref{weaksolution3}. 
To this end, we first need some a priori estimates that are 
independent of $m$, that is given in the next theorem.
%----------------------------------------------------------------------------------------%
%-----    Theorem :  Galerkin solution and A Priori estimate     ------%
%----------------------------------------------------------------------------------------%
\begin{theorem}
If 
$u^0\in V,\,v^0\in H,\,g\in W_1^1((0,T);H_{\Gamma_\text{N}}),\, 
f\in L_2((0,T);H)$, there is a constant 
$C=C(\Omega, \gamma_1, \gamma_2,\rho,T)$ such that,
\begin{equation}   \label{aprioriestimate}
  \begin{split}
    \|u_m\|_{L_\infty((0,T);V)} 
      &+ \|\dot u_m\|_{L_\infty((0,T);H)}
        + \|\ddot u_m\|_{L_2((0,T);V^*)}\\
      &\le C \big\{ \|u^0\|_V+ \|v^0\| 
        + \|g\|_{W_1^1((0,T);H_{\Gamma_\text{N}})}
        + \|f\|_{L_2((0,T);H)}\big\} .
  \end{split}
\end{equation}
\end{theorem}
%--------------------    PROOF    -------------------------%
\begin{proof}
We organize our proof in $2$ steps. 

1. First, we prove the estimate \eqref{aprioriestimate} for 
$u_m$ and $\dot u_m$, that is based on a standard energy method.  
Since $\beta_i(t-s)=D_s\xi_i(t-s)$ and $\xi_i(0)=\gamma_i$ , by \eqref{xiproperties}, 
we first write \eqref{weakGalerkin}, after partial integration in time, as
\begin{equation*}  % \label{weakGalerkin}
  \begin{split}
    \rho (\ddot u_m(t),\vp_k) &+ a(u_m(t),\vp_k)
     -\sum_{i=1}^2 \gamma_i a_i(u_m(t),\vp_k)\\
    &+\sum_{i=1}^2 \int_0^t\!\xi_i(t-s) a_i(\dot u_m(s),\vp_k)\,ds \\
    &\!\!\!\!\!\!\!= (f(t),\vp_k) + (g(t),\vp_k)_{\Gamma_\text{N}}\\
    &-\sum_{i=1}^2 \xi_i(t)a_i(u_m(0),\vp_k) ,
    \quad k=1,\,\dots,\,m ,\,t\in (0,T).
  \end{split}
\end{equation*}
Then multiplying by $\dot d_k(t)$, summing over $k=1,\dots,m$, 
and integrating with respect to $t$, we have
\begin{equation*} 
  \begin{split}
    \rho \|\dot u_m(t)\|^2 + &(1-\bar{\gamma})\|u_m(t)\|_V^2
      +2\sum_{i=1}^2\int_0^t\!\int_0^r\!\xi_i(r-s)
      a_i(\dot u_m(s),\dot u_m(r))\,ds\,dr\\
    &\le \rho\|\dot u_m(0)\|^2 + (1-\underline{\gamma})\|u_m(0)\|_V^2\\
    &\quad +2\int_0^t\!(f(r),\dot u_m(r))\, dr
      +2\int_0^t\! (g(r),\dot u_m(r))_{\Gamma_\text{N}}\,dr\\
    &\quad-2\sum_{i=1}^2\int_0^t\!
    \xi_i(r)a_i(u_m(0),\dot u_m(r))\,dr,
  \end{split}
\end{equation*}
where $\bar \gamma = \max\{\gamma_1,\,\gamma_2\}$ 
and $\underline \gamma = \min\{\gamma_1,\,\gamma_2\}$. 
We note that $0< \underline \gamma, \bar\gamma <1$. 
Since $\xi_i,\,i=1,\,2$ are positive type kernels, 
recalling \eqref{positivetype}, the third term of the left hand side is non-negative. 
Then integration by parts in the last two terms at the right 
side yields
\begin{equation*}  
  \begin{split}
    \rho\|\dot u_m(t)\|^2 &+(1- \bar{\gamma})\|u_m(t)\|_V^2\\
    &\le \rho\|\dot u_m(0)\|^2 +(1- \underline{\gamma})\|u_m(0)\|_V^2
     +2\int_0^t\!(f(r),\dot u_m(r))\, dr\\
    &\quad -2\int_0^t\!(\dot g(r),u_m(r))_{\Gamma_\text{N}}\,dr
      +2(g(t),u_m(t))_{\Gamma_\text{N}}
      -2(g(0),u_m(0))_{\Gamma_\text{N}}\\
    &\quad -2\sum_{i=1}^2\int_0^t\!
    \beta_i(r)a_i(u_m(0),u_{m}(r))\,dr\\
    &\quad -2\sum_{i=1}^2\xi_i(t)a_i(u_m(0),u_{m}(t))
      +2\sum_{i=1}^2\xi_i(0)a_i(u_m(0),u_{m}(0)).
  \end{split}
\end{equation*}
This, using the Cauchy-Schwarz inequality, the trace theorem, 
$\|\beta_i\|_{L_1(\mathbb{R}^+)}=\gamma_i$, 
$\xi_i(t)\le \xi_i(0)=\gamma_i$, 
and \eqref{ai}, implies
\begin{equation*}
  \begin{split}
    \rho\|&\dot u_m(t)\|^2 +(1- \bar{\gamma})\|u_m(t)\|_V^2\\
    &\le \rho\|\dot u_m(0)\|^2 + (1-\underline{\gamma})\|u_m(0)\|_V^2\\
    &\quad +2/C_1\max_{0\le r\le t}\|\dot u_m(r)\|^2
      +C_1\Big(\int_0^t\!\|f(r)\|\,dr\Big)^2\\
    &\quad +2C_{\rm Trace}/C_2\max_{0\le r\le t}\| u_{m}(r)\|_V^2
      +2C_{\rm Trace}C_2\Big(\int_0^t\!\|
      \dot g(r)\|_{H^{\Gamma_\text{N}}}\Big)^2\\
    &\quad +2C_{\rm Trace}/C_3\| u_{m}(t)\|_V^2
      +2C_{\rm Trace}C_3\|g(t)\|_{H_{\Gamma_\text{N}}}^2\\
    &\quad +2C_{\rm Trace}/C_4\| u_{m}(0)\|_V^2
      +2C_{\rm Trace}C_4\| g(0)\|_{H_{\Gamma_\text{N}}}^2\\
    &\quad +2/C_5\Big(\sum_{i=1}^2 \gamma_i\Big)\| u_{m}(0)\|_V^2
      +2C_5\Big(\sum_{i=1}^2\gamma_i\Big)
        \max_{0\le r\le t}\| u_{m}(r)\|_V^2\\
    &\quad +2/C_6\Big(\sum_{i=1}^2\gamma_i\Big)\| u_{m}(0)\|_V^2
      +2C_6\Big(\sum_{i=1}^2\gamma_i\Big)\| u_{m}(t)\|_V^2
      +2\Big(\sum_{i=1}^2\gamma_i\Big)\| u_{m}(0)\|_V^2 .
  \end{split}
\end{equation*}
Hence, considering the facts that $C_{\rm Trace}=C(\Omega)$, 
$\|\dot u_m(0)\|\le\|v^0\|$, and $\|u_m(0)\|_V\le \|u^0\|_V$, 
for some constant $C=C(\Omega,\gamma_1,\gamma_2,\rho, T)$, we have
\begin{equation*}
  \begin{split}
    \|\dot u_m\|_{L_\infty((0,T);H)}^2
    &+\|u_m\|_{L_\infty((0,T);V)}^2\\
    &\le C \big\{\|v^0\|^2+\|u^0\|_V^2
      +\|g\|_{L_\infty((0,T);H_{\Gamma_\text{N}})}^2\\
    &\qquad\quad
      +\|f\|_{L_1((0,T);H)}^2
      +\|\dot g\|_{L_1((0,T);H_{\Gamma_\text{N}})}^2
        \big\}.
  \end{split}
\end{equation*}
This, and the facts that 
$\|g\|_{L_\infty((0,T);H_{\Gamma_\text{N}})}\le C 
\|g\|_{W_1^1((0,T);H_{\Gamma_\text{N}})}$, by Sobolev inequality, 
and $\|f\|_{L_1((0,T);H)}\le C\|f\|_{L_2((0,T);H)}$, imply
\begin{equation}   \label{apriori1}
  \begin{split}
    \|\dot u_m&\|_{L_\infty((0,T);H)}^2
    +\|u_m\|_{L_\infty((0,T);V)}^2\\
    &\le C \big\{\|v^0\|^2+\|u^0\|_V^2
      +\|g\|_{W_1^1((0,T);H_{\Gamma_\text{N}})}^2
      +\|f\|_{L_2((0,T);H)}^2 \big\}.
  \end{split}
\end{equation}

2. Now we need to find a bound for $\ddot u_m$, 
that is by duality with a suitable decomposition of the test 
functions $v\in V$. 
For any fixed $v\in V$ with 
$\|v\|_V\le 1$, we write $v=v^1+v^2$, where 
$v^1\in \Span\{\vp_j\}_{j=1}^m,\,v^2\in \Span(\{\vp_j\}_{j=1}^m)^\perp$. 
We note that $\|v^1\|_V \le 1$. 
 Then from \eqref{weakGalerkin} we obtain,
\begin{equation*}
  \begin{split}
    \rho\langle \ddot u_m(t),v\rangle = \rho(\ddot u_m(t),v^1)
    &= (f(t),v^1)+(g(t),v^1)_{\Gamma_\text{N}}-a(u_m(t),v^1)\\
    &\quad +\sum_{i=1}^2\int_0^t\!\beta_i(t-s)a_i(u_m(s),v^1)\,ds,
  \end{split}
\end{equation*}
that, using the Cauchy-Schwarz inequality,  
 the trace theorem, and \eqref{ai}, implies
 \begin{equation*}
    |\langle \ddot u_m(t),v\rangle| 
    \le \frac{1}{\rho} \Big(
    \|f(t)\|+C_{\rm Trace}\|g(t)\|+\|u_m(t)\|_V
    +\max_{0 \le s \le t} \|u_m(s)\|_V \sum_{i=1}^2\gamma_i
    \Big).
\end{equation*}
This, using \eqref{apriori1}, in a standard way implies
\begin{equation*}
  \begin{split}
    \|\ddot u_m\|_{L_2((0,T);V^*)}^2 
    &\le C\big\{ \|f\|_{L_2((0,T);H)}^2
      +\|g\|_{L_2((0,T);H_{\Gamma_\text{N}})}^2\\
    &\quad\quad+\|v^0\|^2+\|u^0\|_V^2
      +\|g\|_{W_1^1((0,T);H_{\Gamma_\text{N}})}^2
      \big\}.
  \end{split}
\end{equation*}
Therefore, for some constant $C=C(\Omega,\gamma_1,\gamma_2,\rho, T)$,
\begin{equation*}
    \|\ddot u_m\|_{L_2((0,T);V^*)}^2
    \le C\big\{ \|v^0\|^2+\|u^0\|_V^2
      +\|g\|_{W_1^1((0,T);H_{\Gamma_\text{N}})}^2
      +\|f\|_{L_2((0,T);H)}^2
    \big\}.
\end{equation*}
This and \eqref{apriori1} imply  
the estimate \eqref{aprioriestimate}, and the proof is complete.
\end{proof}

%--------   Remark     --------%
\begin{rem}
We note that Lemma 1 and Theorem 1 are also applied  to the simplified 
model problem \eqref{strongform2} in Remark 1. 
That is, there exists a unique function $u_m$ of the form 
\eqref{u_m} satisfying
\begin{equation*}  
  \begin{split}
    \rho (\ddot u_m(t),\vp_k&) + a(u_m(t),\vp_k)
    -\int_0^t\! \beta(t-s) a(u_m(s), \vp_k)\,ds\\
    &= (f(t),\vp_k) + (g(t),\vp_k)_{\Gamma_\text{N}},
    \quad k=1,\,\dots,\,m ,\,t\in (0,T),
  \end{split}
\end{equation*}
with initial conditions \eqref{weakGalerkininitial}. 
Moreover, the a priori estimate \eqref{aprioriestimate} 
still holds with $C=C(\Omega, \gamma, \rho,T)$.
\end{rem}
%---------------------------------------------------------------------------------------%
\subsection{Existence and uniqueness of the weak solution}
%---------------------------------------------------------------------------------------%
First we use Theorem 1, and pass to limits $m \to \infty$, 
to prove existence a weak solution of \eqref{strongform}, 
that is a solution of \eqref{weaksolution1}--\eqref{weaksolution3}. 
Then we prove uniqueness in Theorem 3.

%---------------------------------------------------------------------------------------%
%   Theorem : Existence of  a weak solution
%---------------------------------------------------------------------------------------%
\begin{theorem}
If 
$u^0\in V,\,v^0\in H,\,g\in W_1^1((0,T);H_{\Gamma_\text{N}}),\, 
f\in L_2((0,T);H)$, there exists a weak 
solution of \eqref{strongform}. 
\end{theorem}
%----------------------------          PROOF         ------------------------------%
\begin{proof}
We need to show that there is a solution of 
\eqref{weaksolution1}-\eqref{weaksolution3}, that is a 
weak solution of \eqref{strongform}. 
We organize our proof in 4 steps. 

1. First we note that the estimate \eqref{aprioriestimate} 
does not depend on $m$, so we have
\begin{equation*}
  \begin{split}
    \|u_m\|_{L_\infty((0,T);V)} 
      &+ \|\dot u_m\|_{L_\infty((0,T);H)}
        + \|\ddot u_m\|_{L_2((0,T);V^*)}\\
      &\le K=K(\Omega,\gamma_1,\gamma_2,T,u^0,v^0,f,g) .
  \end{split}
\end{equation*}
This means that the sequences 
$\{u_m\}_{m=1}^\infty$, $\{\dot u_m\}_{m=1}^\infty$, 
$\{\ddot u_m\}_{m=1}^\infty$ are bounded,
\begin{equation}   \label{boundedness}
  \begin{split}
    &\{u_m\}_1^\infty \text{ is bounded in }
      L_\infty((0,T);V)\subset L_2((0,T);V),\\
    &\{\dot u_m\}_1^\infty \text{ is bounded in }
      L_\infty((0,T);H)\subset L_2((0,T);H),\\
    &\{\ddot u_m\}_1^\infty \text{ is bounded in }
      L_2((0,T);V^*).
  \end{split}
\end{equation}
2. Now we prove that the sequence $\{u_m\}_{m=1}^\infty$ passes 
to a limit that satisfies \eqref{weaksolution1}.  
From \eqref{boundedness} and a classical result in 
functional analysis, we conclude that the sequences 
$\{u_m\}_{m=1}^\infty$, $\{\dot u_m\}_{m=1}^\infty$, 
$\{\ddot u_m\}_{m=1}^\infty$ are weakly precompact. 
That is, there are subsequences of 
$\{u_m\}_{m=1}^\infty$, $\{\dot u_m\}_{m=1}^\infty$, 
$\{\ddot u_m\}_{m=1}^\infty$, 
such that
\begin{align}  \label{weakconvergence}
  \begin{aligned}
    &u_l \rightharpoonup u&&\text{in}\quad L_2((0,T);V),\\   
    &\dot u_l \rightharpoonup \dot u&&\text{in}\quad L_2((0,T);H),\\
    &\ddot u_l \rightharpoonup \ddot u&&\text{in}\quad L_2((0,T);V^*),
  \end{aligned}
\end{align}
where the index $l$ is a replacement of the label of the subsequences and '$\rightharpoonup$' denotes weak convergence.
 Consequently, the limit function $u$ satisfies \eqref{weaksolution1}. 
 So it remains to verify \eqref{weaksolution2} and \eqref{weaksolution3}. 

3. To show \eqref{weaksolution2} we fix a positive integer $N$ and we choose
 $v\in\cC([0,T];V)$ of the form
\begin{equation}   \label{v1}
  v(t)=\sum_{j=1}^N h_j(t)\vp _j.
\end{equation}
Then we take  $l\ge N$ and by \eqref{weakGalerkin} we have
\begin{equation}   \label{weakGalerkinSubseq}
  \begin{split}
    \int_0^T\Big(\rho\langle \ddot u_l,v\rangle
    +a(u_l,v)-\sum_{i=1}^{2}\int_0^t\!
      \beta_i(t-s)&a_i(u_l(s),v)\,ds\Big)\,dt\\
    &=\int_0^T\!\big((f,v)+(g,v)_{\Gamma_\text{N}}\big)\,dt.
  \end{split}
\end{equation}
This, by \eqref{weakconvergence}, implies in the limit 
\begin{equation}   \label{weakGalerkinSubseqLimit}
  \begin{split}
    \int_0^T\Big(\rho\langle \ddot u,v\rangle
    +a(u,v)-\sum_{i=1}^{2}\int_0^t\!
    \beta_i(t-s)&a_i(u(s),v)\,ds\Big)\,dt\\
    &=\int_0^T\!\big((f,v)+(g,v)_{\Gamma_\text{N}}\big)\,dt.
  \end{split}
\end{equation}
Since functions of the form \eqref{v1} are dense in $L_2((0,T);V)$, 
this equality then holds for all functions $v\in L_2((0,T);V)$, 
and further it implies \eqref{weaksolution2}. 
 
4. Finally, we need to show that $u$ satisfies the initial conditions   \eqref{weaksolution3}. 
Let $v\in\cC^2([0,T];V)$ be any function with $v(T)=\dot v(T)=0$. 
Then by partial integration in \eqref{weakGalerkinSubseq} we have
\begin{equation*}
  \begin{split}
    \int_0^T\Big(\rho\langle u_{l},\ddot v\rangle&+a(u_l,v)
      -\sum_{i=1}^{2}\int_0^t\!
      \beta_i(t-s)a_i(u_l(s),v)\,ds\Big)\,dt\\
    &=\int_0^T\!\big((f,v)+(g,v)_{\Gamma_\text{N}}\big)\,dt
      -\rho(u_l(0),\dot v(0))+\rho(\dot u_l(0),v(0)) ,
  \end{split}
\end{equation*}
so that, recalling \eqref{weakconvergence} and 
 \eqref{weakGalerkininitial}, in the limit we conclude,
\begin{equation*}
  \begin{split}
    \int_0^T\Big(\rho\langle u,\ddot v\rangle&+a(u,v)
      -\sum_{i=1}^{2}\int_0^t\!\beta_i(t-s)a_i(u(s),v)\,ds\Big)\,dt\\
    &=\int_0^T\!\big((f,v)+(g,v)_{\Gamma_\text{N}}\big)\,dt
      -\rho(u^0,\dot v(0))+\rho(v^0,v(0)) .
  \end{split}
\end{equation*}
On the other hand integration by parts in \eqref{weakGalerkinSubseqLimit}
 gives,
\begin{equation*}
  \begin{split}
    \int_0^T\Big(\rho\langle u,\ddot v\rangle&+a(u,v)
      -\sum_{i=1}^{2}\int_0^t\!\beta_i(t-s)a_i(u(s),v)\,ds\Big)\,dt\\
    &=\int_0^T\!\big((f,v)+(g,v)_{\Gamma_\text{N}}\big)\,dt
      -\rho(u(0),\dot v(0))+\rho(v(0),v(0)) .
  \end{split}
\end{equation*}
Compairing the last two identities we conclude 
\eqref{weaksolution3}, since $v(0),\,\dot v(0)$ are arbitrary. 

Hence $u$ satisfies \eqref{weaksolution1}-\eqref{weaksolution3}, 
that is $u$  is a weak solution of \eqref{strongform}. 
The proof is now complete.
 
\end{proof}

%---------------------------------------------------------------------------------------%
%   Theorem :  Uniqueness of the weak solution
%---------------------------------------------------------------------------------------%
\begin{theorem}
If 
$u^0\in V,\,v^0\in H,\,g\in W_1^1((0,T);H_{\Gamma_\text{N}}),\, 
f\in L_2((0,T);H)$, then the weak solution of \eqref{strongform} is unique. 
\end{theorem}

%----------------     PROOF       ----------------%
\begin{proof}
To prove uniqueness, it is enough to show that $u=0$ is 
the solution of \eqref{weaksolution1}--\eqref{weaksolution3} for 
$u^0=v^0=f=g=0$. Let us fix $r\in [0,T]$ and define
\begin{equation*}
  v(t)=\left\{
  \begin{aligned}
    &\int_t^r\!u(\omega)\, d\omega&&0\le t\le r,\\
    &0&&r\le t\le T.
  \end{aligned}
  \right.
\end{equation*}
We note that 
\begin{equation}   \label{vproperties}
  v(t)\in V,\quad v(r)=0,\quad \dot v(t)=-u(t).
\end{equation}
Then inserting $v$ in \eqref{weaksolution2} and integrating 
with respect to $t$, we have
\begin{equation}   \label{weaksolutionintegral}
  \int_0^r\!\big(\rho \langle \ddot u,v\rangle + a(u,v)\big)\,dt
     -\sum_{i=1}^2 \int_0^r\!\int_0^t\!
     \beta_i(t-s) a_i(u(s), v(t)) \,ds\,dt
        =0.
\end{equation}
For the second term, recalling $-\beta_i(t)=D_t\xi_i(t),\ i=1,2$ from \eqref{xiproperties}, we obtain
\begin{equation*}
  \begin{split}
    -\int_0^r\!\int_0^t\!\beta_i(t-s) a_i(u(s), v(t)) \,ds\,dt
    &=\int_0^r\!\int_s^r\!D_t\xi_i(t-s) a_i(u(s), v(t)) \,dt\,ds\\
    &=\int_0^r\!\xi_i(r-s)a_i(u(s),v(r))\,ds\\
    &\quad -\int_0^r\!\xi_i(0)a_i(u(s),v(s))\,ds\\
    &\quad-\int_0^r\!\int_s^r\!
     \xi_i(t-s) a_i(u(s),\dot v(t)) \,dt\,ds\\
    &=-\gamma_i\int_0^r\!a_i(u(s),v(s))\,ds\\
    &\quad+\int_0^r\!\int_0^t\!\xi_i(t-s) a_i(u(s), u(t)) \,ds\,dt,
  \end{split}
\end{equation*}
where we changed the order of integrals and we used integration 
by parts, $\xi_i(0)=\gamma_i$ from \eqref{xiproperties}, 
and $v(r)=0$ from \eqref{vproperties}. 
Therefore integration by parts in the first term of 
\eqref{weaksolutionintegral} yields
\begin{equation*}
  \begin{split}
    \int_0^r\!\big(-\rho(\dot u,\dot v)+a(u,v)\big)\,dt
    &-\sum_{i=1}^{2}\gamma_i\int_0^r\! a_i(u,v)\,dt \\
    &+\sum_{i=1}^{2}\int_0^r\!\int_0^t\!
      \xi_i(t-s) a_i(u(s), u(t)) \,ds\,dt = 0.
  \end{split}
\end{equation*}
This, using \eqref{vproperties}, implies
\begin{equation*}
  \begin{split}
    \rho\|u(r)\|^2-\rho\|u(0)\|^2 
    &-\|v(r)\|_V^2+\|v(0)\|_V^2
     +\sum_{i=1}^{2}
     \gamma_i \Big(a_i(v(r),v(r))
       -a_i(v(0),v(0))\Big)\\
     &+2\sum_{i=1}^{2}\int_0^r\!\int_0^t\!\xi_i(t-s) a_i(u(s),u(t))\,ds\,dt=0.
  \end{split}
\end{equation*}
Consequently, recalling \eqref{positivetype}, $v(r)=0$, 
$u(0)=0$, $0<\bar\gamma=\max\{\gamma_1,\gamma_2\}< 1$, 
and the fact that $a_i(w,w) \geq 0$, 
we have 
\begin{equation*}
   \rho\|u(r)\|^2+(1-\bar\gamma)\|v(0)\|_V^2\le 0,
\end{equation*}
that implies $u=0$ a.e., and this completes the proof.
\end{proof}

%---------------------------     Remark     -----------------------------%
\begin{rem}
Theorem 2 and Theorem 3 also hold for the simplified problem 
\eqref{strongform2}, see Remark 1 and Remark 2. 
That is, with the assumptions in Theorem 2, 
there exists a unique weak solution for the simplified 
problem.
\end{rem}

%-----------------------------------------------------------------------------%
\section{Regularity}
%------------------------------------------------------------------------------%
Here we study the regularity of the unique weak solution of
\eqref{strongform}, that is, a solution of 
\eqref{weaksolution1}--\eqref{weaksolution3}. 
We explain the limitations for higher regularity in 
Remark 4. We also prove higher regularity of any order of 
the solution of models with smooth kernels in 
Theorem 4.

%%%%%%%%%%%%%%%%%%%
%%%              Corollary               %%%
%%%%%%%%%%%%%%%%%%%
\begin{corollary}
If 
$u^0\in V, v^0\in H, g\in W_1^1((0,T);H_{\Gamma_\text{N}})$, 
and $f\in L_2((0,T);H)$, then for the unique solution $u$ of 
\eqref{weaksolution1}--\eqref{weaksolution3} we have
\begin{equation}   \label{regularity1}
  u\in L_\infty((0,T);V), \quad \dot u\in L_\infty((0,T);H),
    \quad \ddot u\in L_2((0,T);V^*).
\end{equation}
Moreover we have the estimate
\begin{equation}   \label{estimate1}
  \begin{split}
    \|u\|_{L_\infty((0,T);V)} 
      &+ \|\dot u\|_{L_\infty((0,T);H)}
        + \|\ddot u\|_{L_2((0,T);V^*)}\\
      &\le C \big\{ \|u^0\|_V+ \|v^0\| 
        + \|g\|_{W_1^1((0,T);H_{\Gamma_\text{N}})}
        + \|f\|_{L_2((0,T);H)}\big\} .
  \end{split}
\end{equation}
\end{corollary}

%%%   Proof   %%%
\begin{proof}
It is known that if $u_m\rightharpoonup u$, then
\begin{equation*}
  \|u\|\le \lim_{m\to \infty} \inf \|u_m\|. 
\end{equation*}
Then, by \eqref{weakconvergence} and the a priori estimates 
\eqref{aprioriestimate}, we conclude \eqref{regularity1} and 
\eqref{estimate1}.
\end{proof}
We note that, using Remark 3,  Corollary 1 applies also to the simplified problem \eqref{strongform2}.

It is known from the theory of the elliptic operators, 
that global higher spatial regularity can not be obtained with 
mixed boundary conditions. 
Therefore we specialize to the homogeneous Dirichlet boundary 
condition, that is $\Gamma_\tN=\varnothing$, and assume that 
the polygonal domain $\Omega$ is convex. 
We recall the usual Sobolev spaces $H^r=H^r(\Omega)$ and we 
note that here $V=H_0^1(\Omega)$. 
We then use the extension of the operator 
$A$ to an abstract operator  with $\cD(A)=H^2(\Omega)^d\cap V$ 
such that $a(u,v)=(Au,v)$ for sufficiently smooth $u,v$. 
We note that, the elliptic regularity holds, that is, 
\begin{equation}   \label{ellipticregularity}
  \|u\|_{H^2}\le C\|Au\|,\quad u\in H^2(\Omega)^d\cap V.
\end{equation}

%--------------------------------------------------------------------------------%
%-----                         Theorem : regularity                        ------%
%--------------------------------------------------------------------------------%
\begin{theorem}
We assume that $\Gamma_\tN=\varnothing$, and
\begin{equation} \label{betacondition}
  \sum_{i=1}^2\int_0^t \!\beta_i(s)\ ds <1 \ \textrm{or}\ \int_0^t \! \max_{i=1,2} \beta_i(s)\ ds <\frac{1}{2}. 
\end{equation}
If 
$u^0\in H^2,\,v^0\in V$, and $\dot f\in L_2((0,T);H)$, then for the unique solution $u$ of 
\eqref{weaksolution1}-\eqref{weaksolution3} we have
\begin{equation}   \label{regularity2}
  \begin{split}
    &u\in L_\infty((0,T);H^2), 
    \quad \dot u\in L_\infty((0,T);V),\\
    &\ddot u\in L_\infty((0,T);H),
    \quad \dddot u\in L_2((0,T);V^*).
  \end{split}
\end{equation}
Moreover we have the estimate
\begin{equation}   \label{estimate2}
  \begin{split}
    \|u\|_{L_\infty((0,T);H^2)} 
      &+ \|\dot u\|_{L_\infty((0,T);V)}
       + \|\ddot u\|_{L_\infty((0,T);H)}
       + \|\dddot u\|_{L_2((0,T);V^*)}\\
      &\le C \big\{ \|u^0\|_{H^2}+ \|v^0\|_V
       + \|f\|_{H^1((0,T);H)}\big\} .
  \end{split}
\end{equation}
\end{theorem}
%%%%%%%%%%   Proof   %%%%%%%%%%
\begin{proof}
We need to show that estimate \eqref{estimate2} holds for 
the sequence $\{u_m\}_{m=1}^\infty$, and its time derivatives. 
Then, similar to the proof 
of Corollary 1, in the limit we conclude 
\eqref{regularity2} and \eqref{estimate2}.  
We organize our proof in 3 steps. 

1. First we find a bound for the sequences 
$\{\dot u_m\}_{m=1}^\infty$, $\{\ddot u_m\}_{m=1}^\infty$ in 
$L_\infty((0,T);V)$ and $L_\infty((0,T);H)$, respectively. 
Differentiating \eqref{weakGalerkin} with respect to time, 
with  notation $\underline{v}=\dot v$, we have
\begin{equation}   \label{weakGalerkin2}
  \begin{split}
    \rho (\ddot \uu_m(t)&,\vp_k) + a(\uu_m(t),\vp_k)
    - \sum_{i=1}^2\int_0^t\! \beta_i(t-s) a_i(\uu_m(s), \vp_k)\,ds\\
    &= (\uf(t),\vp_k) +\sum_{i=1}^2\beta_i(t)a_i(u_m(0),\vp_k),
    \quad k=1,\,\dots,\,m ,\,t\in (0,T),
  \end{split}
\end{equation}
with the initial conditions
\begin{equation}   \label{weakGalerkininitial2}
  \begin{split}
    &\uu_m(0)=\dot u_m(0)=\sum_{j=1}^m(v^0,\vp_j)\vp _j,\\
    &\dot \uu_m(0)=\ddot u_m(0)
    =\sum_{j=1}^m\big(f(0)-Au_m(0),\vp_j\big)\vp_j.
 \end{split}
\end{equation}
Then, using $\beta_i(t-s)=D_s\xi_i(t-s)$ from \eqref{xiproperties} 
and partial integration in time, we have
\begin{equation*}
  \begin{split}
    \rho (\ddot \uu_m(t),\vp_k) 
    &+ a(\uu_m(t),\vp_k) - \sum_{i=1}^2\gamma_ia_i(\uu_m(t),\vp_k)\\
    &+ \sum_{i=1}^2\int_0^t\! \xi_i(t-s) a_i(\dot \uu_m(s), \vp_k)\,ds\\
    &= (\uf(t),\vp_k) 
     +\sum_{i=1}^2\beta_i(t)a_i(u_m(0),\vp_k)\\
    &\quad-\sum_{i=1}^2\xi_i(t)a(\uu_m(0),\vp_k),
    \quad k=1,\,\dots,\,m ,\,t\in (0,T).
  \end{split}
\end{equation*}
Now, multiplying by $\ddot d_k(t)$, 
summing $k=1,\,\dots,\,m$, and integration with respect to $t$, we have
\begin{equation*} 
  \begin{split}
    \rho\|\dot \uu_m(t)\|^2 
    &+(1-\bar\gamma)\|\uu_m(t)\|_V^2\\
    &+2\sum_{i=1}^2\int_0^t\!\int_0^r\!\xi_i(r-s)
      a_i(\dot \uu_m(s),\dot \uu_m(r))\,ds\,dr\\
    &\le \rho\|\dot \uu_m(0)\|^2 
     +(1-\underline\gamma)\|\uu_m(0)\|_V^2\\
    &\quad+2\int_0^t\!(\uf(r),\dot \uu_m(r))\,dr
     + 2\sum_{i=1}^2\int_0^t\!
      \beta_i(r)a_i(u_m(0),\dot \uu_m(r))\,dr\\
    &\quad-2\sum_{i=1}^2\int_0^t
      \xi_i(r)a_i(\uu_m(0),\dot\uu_m(r))\ dr,
  \end{split}
\end{equation*}
where we recall that 
$\xi_i(0)=\gamma_i$, $\bar\gamma=\max\{\gamma_1,\gamma_2\}$, 
and $\underline\gamma=\min\{\gamma_1,\gamma_2\}$. 
Then, recalling the fact that $\xi_i$ are positive definite 
\eqref{positivetype} 
and integration by parts in the last term, we obtain
\begin{equation*} 
  \begin{split}
    \rho\|\dot \uu_m(t)\|^2 
    &+(1-\bar\gamma)\|\uu_m(t)\|_V^2\\
    &\le \rho\|\dot \uu_m(0)\|^2 
     +(1-\underline\gamma)\|\uu_m(0)\|_V^2\\
    &\quad+2\int_0^t\!(\uf(r),\dot \uu_m(r))\,dr
     + 2\sum_{i=1}^2\int_0^t\!
      \beta_i(r)a_i(u_m(0),\dot \uu_m(r))\,dr\\
    &\quad-2\sum_{i=1}^2\int_0^t
      \beta_i(r)a_i(\uu_m(0),\uu_m(r))\ dr\\
    &\quad-2\sum_{i=1}^2\xi_i(t)a_i(\uu_m(0),\uu_m(t))
     +2\sum_{i=1}^2\xi_i(0)a_i(\uu_m(0),\uu_m(0)),
  \end{split}
\end{equation*}
that, using the Cauchy-Schwarz inequality, 
$\|\beta_i\|_{L_1(\IR^+)}=\gamma_i$, $\xi_i(t)\le \xi_i(0)=\gamma_i$, and \eqref{ai}, 
implies
\begin{equation*} 
  \begin{split}
    \rho\|&\dot \uu_m(t)\|^2 
    +(1-\bar\gamma)\|\uu_m(t)\|_V^2\\
    &\le \rho\|\dot \uu_m(0)\|^2 + (1-\underline\gamma)\|\uu_m(0)\|_V^2\\
    &\quad +2/C_1\max_{0\le r\le t}\|\dot \uu_m(r)\|^2
     +C_1\Big(\int_0^t\!\|\uf(r)\|\,dr\Big)^2\\
    &\quad +2/C_2\Big(\sum_{i=1}^2\gamma_i\Big)\|u_m(0)\|_{H^2}^2
     +2\Big(\sum_{i=1}^2\gamma_i\Big) C_2\max_{0\le r\le t}\|\dot\uu_{m}(r)\|^2 \\
    &\quad +2/C_3\Big(\sum_{i=1}^2\gamma_i\Big)\| \uu_{m}(0)\|_V^2
     +2\Big(\sum_{i=1}^2\gamma_i\Big) C_3 \max_{0\le r\le t}\|\uu_{m}(r)\|_V^2\\
    &\quad+2/C_4\Big(\sum_{i=1}^2\gamma_i\Big)\|\uu_{m}(0)\|_V^2
    +2\Big(\sum_{i=1}^2\gamma_i\Big) C_4 \|\uu_m(t)\|_V^2
    +2\Big(\sum_{i=1}^2\gamma_i\Big)\|\uu_m(0)\|_V^2.
  \end{split}
\end{equation*}
This implies, for some constant $C=C(\gamma_1,\gamma_2,\rho,T)$,
\begin{equation*} 
  \begin{split}
    \|\dot \uu_m\|_{L_\infty((0,T);H)}^2
    &+\|\uu_m\|_{L_\infty((0,T);V)}^2\\
    &\le C\big\{
    \|\dot \uu_m(0)\|^2 + \|\uu_m(0)\|_V^2
    +\|u_m(0)\|_{H^2}^2
    +\|\uf\|_{L_1((0,T);H)}^2\big\}.
  \end{split}
\end{equation*}
Then recalling $\uu=\dot u$, the initial data from 
\eqref{weakGalerkininitial2}, and using 
\begin{equation*}
  \|u_m(0)\|_{H^2}\le \|u^0\|_{H^2},\quad 
  \|\dot u_m(0)\|_V\le \|v^0\|_V,
\end{equation*}
we have
\begin{equation}  \label{regularity:eq1}
  \begin{split}
    \|\ddot u_m\|_{L_\infty((0,T);H)}^2
    &+\|\dot u_m\|_{L_\infty((0,T);V)}^2\\
    &\le C\big\{
    \|u^0\|_{H^2}^2 + \|v^0\|_V^2
    +\|f(0)\|^2 + \|\uf\|_{L_1((0,T);H)}^2\big\}.
  \end{split}
\end{equation}

2. We now find a bound for $\{u_m\}_{m=1}^\infty$ in $L_\infty((0,T);H^2)$. 
We recall the eigenvalue problem \eqref{weakeigenvalue} 
with eigenpairs $\{(\lambda_j,\vp_j)\}_{j=1}^\infty$. 
Then we multiply 
\eqref{weakGalerkin} by $\lambda_k d_k(t)$ and add 
for $k=1,\,\dots,\,m$ to obtain
\begin{equation} \label{ineq0}
  a(u_m,Au_m)=(f-\rho\ddot u_m,Au_m)
  +\sum_{i=1}^2\int_0^t\!\beta_i(t-s)a_i(u_m(s),Au_m(t))\,ds.
\end{equation}
This, using the Cauchy-Schwarz inequality and \eqref{ai}, implies
\begin{equation} \label{ineq}
  \begin{split}
   \|Au_m(t)\|^2
   &\le \frac{2}{\epsilon}
   \Big(\|f(t)\|^2+\rho^2\|\ddot u_m(t)\|^2\Big)
   +\epsilon\|Au_m(t)\|^2\\
   &\quad +\Big(\sum_{i=1}^2\int_0^t\!\beta_i(s)\ ds\Big)\max_{0\le s \le t}\|Au_m(s)\|^2,
 \end{split}
\end{equation}
that, by elliptic regularity \eqref{ellipticregularity} and assumption \eqref{betacondition}, 
gives us
\begin{equation*}
  \|u_m\|_{L_\infty((0,T);H^2)}^2
  \le C\Big(\|f\|_{L_\infty((0,T);H)}^2
  +\|\ddot u_m\|_{L_\infty((0,T);H)}^2\Big).
\end{equation*}
From this and \eqref{regularity:eq1} we conclude
\begin{equation*}
  \begin{split}
    \|\ddot u_m\|_{L_\infty((0,T);H)}^2
    &+\|\dot u_m\|_{L_\infty((0,T);V)}^2
     +\|u_m\|_{L_\infty((0,T);H^2)}^2\\
    &\le C\big\{
    \|u^0\|_{H^2}^2 + \|v^0\|_V^2
    +\|f\|_{L_\infty((0,T);H)}^2 
    + \|\uf\|_{L_1((0,T);H)}^2\big\},
  \end{split}
\end{equation*}
that using 
$\|f\|_{L_\infty((0,T);H)}\le C \|f\|_{W_1^1((0,T);H)}$, 
by Sobolev inequality, we have
\begin{equation*}
  \begin{split}
    \|\ddot u_m\|_{L_\infty((0,T);H)}^2
    &+\|\dot u_m\|_{L_\infty((0,T);V)}^2
     +\|u_m\|_{L_\infty((0,T);H^2)}^2\\
    &\le C\big\{
    \|u^0\|_{H^2}^2 + \|v^0\|_V^2
    +\|f\|_{W_1^1((0,T);H)}^2 \big\}\\
    &\le C\big\{
    \|u^0\|_{H^2}^2 + \|v^0\|_V^2
    +\|f\|_{H^1((0,T);H)}^2 \big\}.
  \end{split}
\end{equation*}

3. Finally from \eqref{weakGalerkin2}, similar to step 2 of 
the proof of Theorem 1, we obtain
\begin{equation*}  %\label{regularity:eq3}
  \begin{split}
    \|\dddot u_m\|_{L_2((0,T);V^*)}^2
    &\le C\big\{
    \|u^0\|_{H^2}^2 + \|v^0\|_V^2
    +\|f\|_{H^1((0,T);H)}^2 \big\}.
  \end{split}
\end{equation*}
The last two estimates then, in the limit, imply \eqref{regularity2} and the 
desired estimate \eqref{estimate2}. The proof is 
now complete.
\end{proof}

\begin{rem}
If we continue differentiating \eqref{weakGalerkin2} in 
time to investigate more regularity, we obtain
\begin{equation*}   %\label{weakGalerkin2}
  \begin{split}
    \rho (\dddot \uu_m(t),\vp_k&) + a(\dot \uu_m(t),\vp_k)
    - \sum_{i=1}^2\int_0^t\! \beta_i(t-s) a_i(\dot \uu_m(s), \vp_k)\,ds\\
    &= (\ddot f(t),\vp_k) +\sum_{i=1}^2\dot \beta_i(t)a_i(u_m(0),\vp_k)\\
    &\quad+\sum_{i=1}^2\beta_i(t)a_i(\uu_m(0),\vp_k),
    \quad k=1,\,\dots,\,m ,\,t\in (0,T),
  \end{split}
\end{equation*}
Further, from 
$\dot \beta_i(t)a_i(u_m(0),\vp_k),\ i=1,2$, 
we get $\dot \beta_i(t)a_i(u_m(0),\ddot \uu_m(t))$, 
but the $\dot\beta_i$ are not integrable. 
Besides, after integration in time, we can not use partial integration to transfer one time derivative from $\dot\beta_i$ 
to $\ddot \uu_m(t)$, since $\beta_i$ is singular at $t=0$. 
This means that we can not get 
more regularity with weakly singular kernels $\beta_i$. 
This also indicates that with smoother kernel we can get 
higher regularity in case of homogeneous Dirichlet 
boundary condition under the appropriate assumption on the 
data, that is, more regularity and compatibility conditions.
\end{rem}

\begin{rem}
For the simplified problem \eqref{strongform2}, the inequality \eqref{ineq} is
\begin{equation*} %\label{ineq}
  \begin{split}
   \|Au_m(t)\|^2
   &\le \frac{2}{\epsilon}
   \Big(\|f(t)\|^2+\rho^2\|\ddot u_m(t)\|^2\Big)
   +\epsilon\|Au_m(t)\|^2\\
   &\quad +\Big(\int_0^t\!\beta(s)\ ds\Big)\max_{0\le s \le t}\|Au_m(s)\|^2.
 \end{split}
\end{equation*}
Hence, the assumption \eqref{betacondition} can be ignored, since 
$\int_0^t\!\beta(s)\ ds<\gamma<1$. That is, Theorem 3 applies also to the simplified problem 
\eqref{strongform2}, ignoring the assumption \eqref{betacondition}.
\end{rem}
\begin{rem}
We recall the definition of the operators $A,A_1$ and $A_2$ from \eqref{Ai}.  
%and the fact that they are self-adjoint, positive definite linear operators. 
If the solution $u$ is regular such that its second order partial derivatives are comutative, 
then the operator $A^{1/2}$ is comutative with the operators $A_1^{1/2},A_2^{1/2}$. 
Here, the operator $A^l\  (l\in \mathbb{R})$ is defined by, see e.g., \cite{Thomee_Book}
\begin{equation*}
  A^l v=\sum_{k=1}^\infty \lambda_k^l (v,\vp_k)\vp_k,
\end{equation*}
where $\{(\lambda_k,\vp_k)\}_{k=1}^\infty$ are the eigenpairs of the operator $A$,  
and in a similar way $A_1^l$ and $A_2^l$ are defined.
In this case the assumption \eqref{betacondition} is replaced by
\begin{equation*}
  \int_0^t \! \max_{i=1,2}\beta_i(s)\ ds <1,
\end{equation*}
since in \eqref{ineq0} we have
\begin{equation*}
  \begin{split}
    \sum_{i=1}^2\int_0^t\! \beta_i(t-s)a_i(u_m(s),Au_m(t))\ ds
    &\le \int_0^t\! \max_{i=1,2}\beta_i(t-s)a(u_m(s),Au_m(t))\ ds\\
    &\le \Big(\int_0^t\! \max_{i=1,2}\beta_i(t-s)\ ds\Big)
      \max_{0\le s\le t}\|Au_m(s)\|^2,
  \end{split}
\end{equation*}
where we used the fact that
\begin{equation*}
  \begin{split}
    a_i(v,Av)
    &=(A^{1/2}A_i^{1/2}A_i^{1/2}v,A^{1/2}v)\\
    &=(A_i^{1/2}A^{1/2}A_i^{1/2}v,A^{1/2}v)=a_i(A^{1/2}v,A^{1/2}v)\ge 0.
  \end{split}
\end{equation*} 
\end{rem}

In the next theorem we state regularity of any order of the solution of 
models with smooth kernels.  The proof is by induction and simillar to 
 the proof of Theorem 4, and we omit the details. 
 
%---------------------------------------------------------------------------------%
%----------------       Theorem : Higher regularity        --------------%
%---------------------------------------------------------------------------------%
\begin{theorem}
We assume that $\Gamma_\tN=\varnothing$, and condition \eqref{betacondition} holds.
Assume $(r=0,1,\dots)$
\begin{equation*}
 \begin{aligned}
  &u^0\in H^{r+1},\quad v^0\in H^r,&\\
  &\frac{d^k f}{dt^k}\in L_2((0,T);H^{r-k}),&\textrm{for}\ \  k=0,\dots,r, \\
  & \ \beta_i \in W^{r-1}_1(0,T),& \textrm{if} \ \  r \geq 2,
 \end{aligned}
\end{equation*}
and the $r^{\textrm{th}}$-order compatibility conditions
\begin{equation*}
 \begin{aligned}
  &u^0_0:=u^0\in V,
   \quad u^0_1:=v^0,&\\
  &u^0_2:=\frac{1}{\rho}(f(0)-Au^0)\in V,&\textrm{if}\ \ r=2\\
  &u^0_r:=\frac{1}{\rho}
   \Big(
   \frac{d^{r-2}}{{dt^{r-2}}}f(0)-Au^0_{r-2}
   +\sum_{j=0}^{r-3}\sum_{i=1}^2\frac{d^j}{dt^j}
    \beta_i(0)A_iu^0_{r-3-j}
   \Big) \in V,&\textrm{if}\ \ r\ge 3.
 \end{aligned}
\end{equation*}
Then for the unique solution $u$ of 
\eqref{weaksolution1}-\eqref{weaksolution3} 
we have
\begin{equation*}
  \frac{d^k}{dt^k} u \in L_\infty ((0,T);H^{r+1-k})
  \quad (k=0,\dots,r+1),
\end{equation*}
and we have the estimate 
\begin{equation*}
  \sum_{k=0}^{r+1} 
  \Big\|\frac{d^k u}{dt^k}\Big\|_{L_\infty ((0,T);H^{r+1-k})}
  \le C \bigg( 
  \sum_{k=0}^r 
  \Big\|\frac{d^k f}{dt^k}\Big\|_{L_2((0,T);H^{r-k})}
  +\|u^0\|_{H^{r+1}}+\|v^0\|_{H^r} \bigg).
\end{equation*}
\end{theorem}

We note that, with $\beta_i \in W^{r-1}_1(0,T)$ we have 
$\beta_i \in \cC^{r-2}[0,T]$ by Sobolev inequality. 
Therefore $u^0_r$ in the compatibility conditions is 
well-defined.

We also note that Remark 5 holds for Theorem 5, too.  
Remark 6 can be applied to Theorem 5, provided the solution $u$ is smooth 
enough such that the operator $A^{\frac{r+1}{2}}$ is comutative with the operators 
$A_i^{1/2},\ i=1,2$, that is, when $(r+2)$-th order partial derivatives of the solution $u$ 
are comutative.

\textbf{Acknowledgment.} 
I would like thank Dr. Milena Racheva and Prof. Mikael Enelund for fruitful discussion 
on fractional order viscoelasticity. 
I also thank Prof. Stig Larsson and 
the anonymous referees for constructive comments.  
%and  the Department of Mathematical Sciences, 
%Chalmers University of Technology and University of Gothenburg, 
%where this work was initiated.

\bibliographystyle{amsplain}
\bibliography{Thesis}

\end{document}